\documentclass[11pt]{amsart}

\usepackage[utf8]{inputenc}
\usepackage[english]{babel}
\usepackage[T1]{fontenc}

\title[Equivalence of Weak Mean Curvature Flows]{Equivalence of weak solution concepts for mean curvature flow}

\author{Tim Laux}
\thanks{T.L.\ Heidelberg University, Im Neuenheimer Feld 205, 69120 Heidelberg, Germany. \nolinkurl{tim.laux@math.uni-heidelberg.de}}

\author{Anton Ullrich}
\thanks{A.U.\ Carnegie Mellon University, 5000 Forbes Avenue, 
Pittsburgh, PA 15213,
United States. \nolinkurl{aullrich@andrew.cmu.edu}}

\date{}

\usepackage{amsthm}
\usepackage{stmaryrd}
\usepackage{enumitem}

\usepackage{algorithm2e}
\SetKwComment{Comment}{/* }{ */}
\RestyleAlgo{ruled}
\usepackage{listings}
\usepackage{verbatim}
\usepackage{stmaryrd}
\usepackage{dsfont}
\usepackage{tikz}
\usetikzlibrary{calc, quotes, angles}
\usepackage{float}
\usepackage{amsmath}
\usepackage{amssymb}
\usepackage{amsthm}
\usepackage{mathrsfs}

\usepackage{mathtools}
\usepackage{esint}
\usetikzlibrary{decorations.pathreplacing}

\PassOptionsToPackage{dvipsnames}{xcolor}
\RequirePackage{xcolor}
\definecolor{webgreen}{rgb}{0,.5,0}
\definecolor{webbrown}{rgb}{.6,0,0}
\definecolor{RoyalBlue}{cmyk}{1, 0.50, 0, 0}

\usepackage[colorlinks,
pdfpagelabels,
pdfstartview = FitH,
bookmarksopen = true,
bookmarksnumbered = true,
linkcolor = RoyalBlue,
plainpages = false,
urlcolor=webbrown,
hypertexnames = false,
citecolor = webgreen] {hyperref}

\newcommand*{\dd}{\mathop{}\!\mathrm{d}}

\DeclareMathOperator*{\BV}{BV}

\DeclareMathOperator*{\esssup}{ess\,sup}

\DeclareMathOperator*{\supp}{supp}

\theoremstyle{plain}
\newtheorem{definition}{Definition}[section]
\newtheorem*{definition*}{Definition}
\newtheorem{proposition}[definition]{Proposition}
\newtheorem{theorem}[definition]{Theorem}
\newtheorem*{theorem*}{Theorem}
\newtheorem{remark}[definition]{Remark}
\newtheorem*{remark*}{Remark}
\newtheorem{lemma}[definition]{Lemma}

\newtheorem*{example*}{Example}
\newtheorem*{conjecture*}{Conjecture}

\usepackage{pgfplots}
\usepackage{todonotes}
\usepackage{url}
\pgfplotsset{compat=newest}
\usepgfplotslibrary{colorbrewer}
\usetikzlibrary{math}
\usetikzlibrary{intersections, pgfplots.fillbetween}
\usetikzlibrary{decorations.markings}
\pgfdeclarelayer{pre main}
\allowdisplaybreaks

\begin{document}

\begin{abstract}
    
    We provide a connection between weak solution concepts of mean curvature flow.
    On the one side we have the viscosity solution which is based on the comparison principle. 
    On the other, variational solutions, which are combined Brakke flows and distributional solutions.
    We prove that if one has a foliation by variational solutions, then the resulting function is the unique viscosity solution. 
    
    This answers an open question suggested by the work of Evans and Spruck~[J.\ Geom.\ Anal., 5(1), 1995] and the authors~[J.\ Geom.\ Anal., 34(12), 2024].
    These results show that almost every level set of the viscosity solution is a variational solution. 
    Thus, we establish the equivalence of these solution concepts.
    Moreover, we show the generic uniqueness of variational solutions.
\end{abstract}


\maketitle

\section{Motivation}

Mean curvature flow is one of the most applied parabolic differential evolution equations. It appears naturally in almost every field of science. For instance, it finds applications in computer visualization, material sciences, mathematics, physics, biology, and machine learning. The flow is defined on submanifolds by relating the mean curvature to the negative normal velocity. This leads to a relaxation of the surface tension. Formally, this can also be seen as a gradient flow of the perimeter functional. 

\smallskip

In general, mean curvature flow can develop singularities in finite time even if started from a smooth initial datum. This shows the necessity of weak solution concepts. Some of the most used ones are viscosity solutions \cite{ChenGigaGoto, ESI}, which are based on the comparison principle, and distributional solutions \cite{luckhaus1995implicit} and Brakke flows \cite{Brakke}, which are based on the gradient flow structure.

\smallskip

Certain singularities result in non-uniqueness, but the relaxation of the solution concept should not lead to further non-uniqueness or instability issues. This famously fails for Brakke flows, which can instantly disappear. However, the viscosity solution \cite{ ChenGigaGoto, ESI} and distributional solutions~\cite{FHLS} are at least consistent with classical solution in the sense that they agree with the unique smooth mean curvature flow as long as the latter exists.
In the present work, we relate weak solution concepts past singularities. 
More precisely, we are interested in the comparison of the viscosity solution on the one hand to distributional solutions and Brakke flows on the other.

\smallskip

In previous works by Evans and Spruck \cite{ESIV} and the authors \cite{genericMCF}, it has been shown that these solution concepts partially imply each other. To be more precise, Evans and Spruck proved that if one considers the viscosity solution to level set mean curvature flow, then almost every level set of this function is a Brakke flow, cf.~\cite[Theorem 7.1]{ESIV}. More recently, the authors have shown that almost every level set of the viscosity solution is also a distributional solution, cf.~\cite[Theorem 1.1]{genericMCF}.
\smallskip

In this paper, we prove the converse in Theorem~\ref{Thm:MainThm} below, which can be restated informally as the following.

\begin{theorem*}[Reverse direction -- equivalence of solutions]
    Let $g$ be a well-prepared initial datum. If almost every level set of $g$ evolves as a distributional solution and a Brakke flow, then this evolution is the unique viscosity solution to mean curvature flow starting at $g$.
\end{theorem*}

Therefore, such a foliation by Brakke flows that are also distributional solutions is in fact equivalent to the viscosity solution.

Furthermore, as a consequence of this, we obtain that generically an evolution that is a Brakke flow and a distributional solution is unique, cf.\ Theorem~\ref{Thm:UniqueVarSol}.

\smallskip

The equivalence of comparison based weak solutions with those based on energy method has been known for a while in the context of partial differential equations. Ishii~\cite{Ishii:ontheequivalence} showed that viscosity solutions and distributional solutions are equivalent in the case of linear degenerately elliptic equations; see also~\cite{JuutinenLindqvistManfredi} for extensions to the $p$-Laplace equation. 
However, in the case of mean curvature flow, this equivalence is more subtle and can only hold \emph{generically} as we show in examples below.

\smallskip

The rest of the paper is organized as follows.
In Section~\ref{Sec:Preliminaries}, we give the precise definitions of the weak solution concepts and list the main properties and results related to them. These results are based on \cite{ChenGigaGoto, ESIV, genericMCF, white2024avoidance}. Afterwards, in Section~\ref{Sec:ProofMainThm}, we prove our main theorem, Theorem~\ref{Thm:MainThm}. The main idea here is to combine the avoidance principle of Brakke flows with the continuity of distributional solutions. Together, this implies a comparison principle, Lemma~\ref{Lem:CompPrinSet}. This then can be used to construct a function from the evolutions that also satisfies the comparison principle. Since viscosity solutions are also distributional solutions, Theorem~\ref{Thm:ViscBV}, and Brakke flows, Theorem~\ref{Thm:ViscBrakke}, this then implies that they have to coincide.

\section{Preliminaries}
\label{Sec:Preliminaries}

The following definitions are taken from the thesis of one of the authors,~\cite{Ullrich_2025}.
We begin with defining the notion of a smooth solution to level set mean curvature flow.

\begin{definition}[Level set mean curvature flow]
We say that a function $u\in C^2(\mathbb{R}^d\times [0,\infty))$ is a solution to level set mean curvature flow starting at $g\in C^2(\mathbb{R}^d)$ if
\begin{align}
    \begin{cases}
    \partial_t u=|\nabla u|\nabla\cdot \frac{\nabla u}{|\nabla u|} &\text{ in }\mathbb{R}^d\times \mathbb{R}_{>0}, |\nabla u|\neq 0,\\
    \partial_t u=0  &\text{ on }\mathbb{R}^d\times \mathbb{R}_{>0}, |\nabla u|=0,\\
    u=g&\text{ on }\mathbb{R}^d\times \{t=0\}.
    \end{cases}
    \label{Eq:Levelset}
\end{align}
\end{definition}

\begin{remark*}
    The mean curvature operator has different equivalent representations as a normalized $1$-Laplacian or the level set Laplacian which can be useful depending on the context:
    $$|\nabla u|\nabla\cdot \frac{\nabla u}{|\nabla u|}=\Delta u-\frac{\nabla u}{|\nabla u|}\cdot \nabla^2u \frac{\nabla u}{|\nabla u|}=\nabla^2 u:(\text{Id}-\nu_u\otimes \nu_u).$$
\end{remark*}

The concept of viscosity solutions is based on the comparison principle. Intuitively, one squeezes the continuous solution $u$ from above and below in smooth functions $\varphi$ which evolve by mean curvature flow in the correct differential inequality.

\begin{definition}[Viscosity solution] 
\label{Def:ViscositySolution}
A continuous function $u\in C(\mathbb{R}^d\times[0,T))$ with initial datum $g\in C(\mathbb{R}^d)$ whose level sets are compact is called a viscosity super-solution to \eqref{Eq:Levelset} if, for any $\varphi\in C^\infty(\mathbb{R}^d\times[0,T))$ and $(x_0,t_0)\in \mathbb{R}^d\times(0,T)$ such that $u-\varphi$ has a local minimum at $(x_0,t_0)$, the following inequality holds,
\begin{enumerate}[label=\roman*)]
    \item for non-critical points $\nabla \varphi(x_0,t_0)\neq 0:$
    $$\partial_t \varphi\geq \Delta \varphi-\frac{\nabla \varphi}{|\nabla \varphi|}\cdot \nabla^2\varphi\frac{\nabla \varphi}{|\nabla \varphi|}\quad\text{at }(x_0,t_0)$$
    \item and for critical points $\nabla \varphi(x_0,t_0)=0$ there exists $\xi\in\mathbb{R}^d, |\xi|\leq 1:$
    $$\partial_t \varphi\geq \Delta \varphi-\xi\cdot \nabla^2\varphi \xi\quad\text{at }(x_0,t_0).$$
\end{enumerate}
We say that $u$ is a viscosity sub-solution if $-u$ is a viscosity super-solution. Finally, $u$ is called a viscosity solution if it is both a viscosity sub- as well as a viscosity super-solution.
\end{definition}

The concept of distributional solutions is based on test functions to formulate the mean curvature equation $V=-H$ in a distributional form.

\begin{definition}[$\BV$ solution, distributional solution, cf.~\cite{luckhaus1995implicit}]
\label{Def:BVSol}
A family of finite perimeter sets $(\Omega(t))_{t\in [0, T)}$ such that $\bigcup\limits_{t\in [0,T)}\Omega(t)$ is measurable in space and time is called distributional solution to mean curvature flow if the sets admit a uniform perimeter bound
\begin{align}
    \esssup\limits_{t\in [0,T)}P(\Omega(t))&<\infty\label{Eq:LSP}
\end{align}
and there is a $|\mu_{\Omega(t)}|\dd t$-measurable $L^2$-bounded function $V\colon\mathbb{R}^d\times (0,T)\to \mathbb{R}$
\begin{align}
    \int\limits_{\mathbb{R}^d\times (0,T)}V^2\dd |\mu_{\Omega(t)}|\dd t&<\infty\label{Eq:LSV}
\end{align}
such that the following holds:
\begin{enumerate}[label=\roman*)]
    \item The function $V(\cdot, t)$ is the normal velocity of $\Omega(t)$ in the sense that for each test function $\zeta\in C_c^1(\mathbb{R}^d\times [0,T))$ it satisfies
    \begin{align}
        \int\limits_0^T\int\limits_{\Omega(t)}\partial_t\zeta\dd x\dd t&=-\int\limits_{\mathbb{R}^d\times(0,T)}\zeta V\dd |\mu_{\Omega(t)}|\dd t-\int\limits_{\Omega_0}\zeta(x, 0)\dd x\label{Eq:LSdistV}.
    \end{align}
    \item For any test vector field $\xi\in C_c^1(\mathbb{R}^d\times(0,T);\mathbb{R}^d)$ it holds
    \begin{align}
        \int\limits_{\mathbb{R}^d\times(0,T)} (\nabla\cdot \xi-\nu\cdot \nabla \xi\nu) \dd |\mu_{\Omega(t)}|\dd t&=-\int\limits_{\mathbb{R}^d\times(0,T)} V\xi\cdot \nu\dd |\mu_{\Omega(t)}|\dd t.\label{LSdistMC}
    \end{align}
\end{enumerate}
\end{definition}

Similarly to distributional solutions, Brakke flows generalize the evolving manifolds to define a weak solution concept for mean curvature flow. Where distributional solution represented the manifold as the reduced boundary of a set, Brakke flows generalize the manifolds to varifolds. This is then combined with a weak definition of the mean curvature via the first variation of the varifold and a gradient flow inequality.

\begin{definition}[Brakke flow]
    \label{Def:BrakkeFlow}
    An $m$-dimensional integral Brakke flow is a family of integral $m$ dimensional varifolds $\mathcal{V}_t$ in $\mathbb{R}^d$ with associated mass measure $\mu(t)$ such that for a.e.\ $t,t_1,t_2\in [0,T)$ with $t_1<t_2$
    \begin{enumerate}
        \item The first variation can be representated by a locally $L^2$-integrable function $\vec{H}(\cdot,t)\in L^2_{loc}(\mu(t))$ such that for all $X\in C^1_c(\mathbb{R}^d)$
        $$\delta \mathcal{V}_t(X)=\int\limits_{\mathbb{R}^d} \vec{H}(\cdot,t)\cdot X\dd\mu(t),$$
        \item For $f\in C_c^1(\mathbb{R}^d\times [t_1,t_2];\mathbb{R}_{\geq 0})$, we have
        $$\int\limits_{\mathbb{R}^d} f(\cdot, t_2)\dd \mu(t_2)+\int\limits_{t_1}^{t_2} \int\limits_{\mathbb{R}^d} \left(|\vec{H}|^2f+\vec{H}\cdot\nabla f-\partial_t f\right)\dd\mu(t)\dd t\leq \int\limits_{\mathbb{R}^d} f(\cdot, t_1)\dd \mu(t_1).$$
    \end{enumerate}
\end{definition}

Now, that we have defined the weak solution concepts, we relate them to each other. For this, we need to restrict the class of initial data.

\begin{definition}[Well-prepared initial datum {\cite{ESIV}}]
\label{def:wellprepared}
We call an initial level set datum $g$ \emph{well-prepared} if $g\in C^3(\mathbb{R}^d)$, $g$ is constant outside of some large ball, and $g$ satisfies
\begin{align}
    \sup_{0<\varepsilon\leq 1} \int\limits_{\mathbb{R}^d} \Big|\nabla \cdot \Big(\frac{\nabla g}{\sqrt{|\nabla g|^2 + \varepsilon^2}}\Big) \Big| \dd x <\infty.
\end{align}
\end{definition}

\begin{remark*}
    By \cite[Lemma 2.1]{ESIV}, one can construct a well-prepared initial datum $g$
for any given compact smooth $0$-level set.
\end{remark*}

\begin{theorem}[Viscosity solutions are distributional solutions, {\cite[Theorem 1.1]{genericMCF}}]
\label{Thm:ViscBV}
Let $u\in C(\mathbb{R}^d\times [0,T))$ be a viscosity solution to mean curvature flow in the sense of Definition~\ref{Def:ViscositySolution} with well-prepared initial data according to Definition~\ref{def:wellprepared}. 
Then almost every level set of $u$ is a BV solution to mean curvature flow in the sense of Definition~\ref{Def:BVSol}.
\end{theorem}

\begin{theorem}[Viscosity solutions are Brakke flows {\cite[Theorem 7.1]{ESIV}}]
    \label{Thm:ViscBrakke}
    Let $u(x,t)$ be a viscosity solution of mean curvature flow starting from a well-prepared initial datum $g$. Then almost every level set of $u$ is as a unit density Brakke flow.
\end{theorem}

Based on the previous two theorems, we define the concept of a variational solution.
\begin{definition}[Variational solution]
\label{Def:VarSol}
    We call the evolution of a measurable set $\Omega(t)$ a variational solution if $\Omega(t)$ is a distributional solution and $\Sigma(t)\coloneqq\partial^*\Omega(t)$ is a Brakke flow. 
\end{definition}

\begin{remark}
\label{Rem:Boundary}
    Due to the previous theorems, one can show that almost every level set of a viscosity solution with well-prepared initial datum is a variational solution.

    The only thing that remains to show is that the level sets $\Sigma(t)$ are the reduced boundary of the corresponding sub level sets $\Omega(t)$, i.e.,
    $$\Sigma(t)=\partial^*\Omega(t)$$ for almost all times $t$. Here, we can use that the viscosity solution with well-prepared initial datum is Lipschitz continuous in space and time by the comparison principle. Therefore, one can apply~\cite[Theorem 2.5]{albertistructure}. This implies that $$\mathcal{H}^{d}(\{\nabla u(\cdot,\cdot)=0\}\cap\{u(\cdot,\cdot)=s\})=0$$ 
    and $\Sigma_s$ is $d$-rectifiable (in time and space) with $\mathcal{H}^{d}(\Sigma_s)<\infty$. Moreover, the derivative $\nabla u$ exists $\mathcal{H}^d$-almost everywhere on $\Sigma_s$ and is non zero. Thus, for almost all times $t$ and level sets $s$ we indeed have $\partial^*\Omega(t)=\Sigma(t)$.
\end{remark}

The key properties which are essential for the proof are the comparison and avoidance properties as well as an $L^1$-continuity.

\begin{proposition}[Comparison principle for viscosity solutions {\cite{ChenGigaGoto}}]
    \label{Prop:MaxPrinciple}
    Let $u,v\in C(\mathbb{R}^d\times [0,T))$ be viscosity solutions with two initial condition $u_0,v_0\in C(D)$ such that $u_0\leq v_0$. Then, the evolutions remain ordered for all times $t \geq 0$, i.e., the evolutions satisfy $u\leq v$.
\end{proposition}

\begin{proposition}[Avoidance principle for Brakke flows {\cite[Theorem 10.6]{IlmanenBook}, \cite[Theorem 1]{white2024avoidance}}]
    \label{Prop:Avoidance}
    Let $\Sigma(t),\widetilde\Sigma(t)$ be Brakke flows with initial conditions $\Sigma_0,\widetilde\Sigma_0$ that are disjoint. The evolution then fulfills the avoidance principle, i.e., for all times $t \geq 0$, the sets are disjoint. Moreover, it holds
    $$d(\Sigma(t),\widetilde\Sigma(t))\geq d(\Sigma_0,\widetilde\Sigma_0).$$
\end{proposition}

\begin{lemma}[$L^1$-continuity in time of BV solutions {\cite[Lemma 4.7]{genericMCF}}]
    Let $(\Omega(t))_{t\in [0,T)}$ be a distributional solution to mean curvature flow according to Definition~\ref{Def:BVSol}. Then, there exists a representative of $\Omega(t)$ that is continuous in time in the $L^1$-norm, i.e., for $t\in[0,T)$ and $t_n\to t$ we have $$|\Omega(t_n)\Delta\Omega(t)|\to 0.$$
    \label{Lemma:L1ContBV}
    \vspace*{-1em}
\end{lemma}

\begin{remark*}
    This version is a modification of the original lemma using that the proof provides a $\frac{1}{2}$-Hölder continuity in time.
\end{remark*}

Now we can formulate our main result precisely.

\begin{theorem}[Reverse direction -- equivalence of solutions]
\label{Thm:MainThm}
    Let $g$ be a well-prepared initial datum in the sense of Definition~\ref{def:wellprepared}. Assume that, for almost all $s\in\mathbb{R}$, the sets $\Omega_s(t)$ are variational solutions on $(0,T)$ starting from the respective sub level set $\{g<s\}$ of $g$, cf.\ Definition~\ref{Def:VarSol}.
    
      Then there exists a function $v\in C(\mathbb{R}^d\times[0,T))$ whose sub level sets are $\Omega_s(t)$. Moreover, this function is the viscosity solution starting at $g$.
\end{theorem}

\begin{figure}
    \centering
    \begin{tikzpicture}
    \begin{axis}[
        xlabel={$x$},
        xmin=-2.1, xmax=2.1,
        ymin=-0.5, ymax=4.5,
        xtick={-2,-1,0,1,2},
        major grid style={gray!50,thick},
        minor grid style={gray!20,dotted},
        axis lines=middle,
        enlargelimits=0.1,
        width=0.4\textwidth,
    ]
    \addplot[domain=-2:2, samples=30, color=blue, thick] {x^2};
    \end{axis}
    
    \begin{axis}[
        xlabel={$x$},
        xmin=-2.1, xmax=2.1,
        ymin=-0.5, ymax=4.5,
        major grid style={gray!50,thick},
        minor grid style={gray!20,dotted},
        axis lines=middle,
        xtick={-2,-1,0,1,2},
        enlargelimits=0.1,
        width=0.4\textwidth,
        xshift=0.35\textwidth,
    ]
    \addplot[domain=-sqrt(2.5):sqrt(2.5), samples=30, color=blue, thick] {x^2+1.5};
    \end{axis}
    \begin{axis}[
        xlabel={$x$},
        xmin=-2.1, xmax=2.1,
        ymin=-0.5, ymax=4.5,
        xtick={-2,-1,0,1,2},
        major grid style={gray!50,thick},
        minor grid style={gray!20,dotted},
        axis lines=middle,
        enlargelimits=0.1,
        width=0.4\textwidth,
        xshift=0.7\textwidth,
    ]
    \addplot[domain=-sqrt(2):-1, samples=10, color=blue, thick] {x^2+2};
    \addplot[domain=-1:1, samples=10, color=blue, thick] {3};
    \addplot[domain=1:sqrt(2), samples=10, color=blue, thick] {x^2+2};
    \end{axis}
    
    \end{tikzpicture}
    \caption{A sliced view of the example $u(x,t)$ in Equation~\eqref{Eq:Exu} for times $t=0,t<1$ and $t\geq 1$.}
    \label{Fig:ExSqr}
\end{figure}
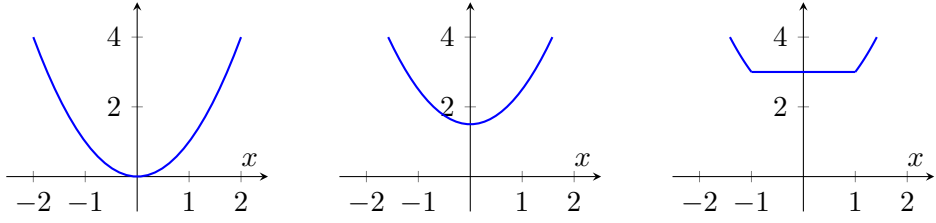

\begin{remark*}
    In fact, besides Brakke's inequality which is essential to our proof, from the distributional solution we only need the $L^1$-continuity in time, i.e., Lemma~\ref{Lemma:L1ContBV}.
    In particular, our main result also applies to the two-phase version of canonical Brakke flows~\cite[Theorem 2.11]{StuvardTonegawa}, or Brakke flows that are also $L^2$-flows~\cite{mugnairoeger} or De Giorgi solutions~\cite{HenselLauxDeGiorgi, LauxOttoThreshMCFDeGiorgiMM}.

    This continuity property in time is necessary since Brakke flows can spontaneously vanish. We demonstrate this with the following example of shrinking spheres that disappear at an instance in time. Consider the function 
    \begin{align}
        \label{Eq:Exu}u(x,t)\coloneqq \begin{cases}
    x_1^2+x_2^2+2t&\text{for }x\in\mathbb{R}^2,t\in [0,1),\\
    \max(x_1^2+x_2^2+2t,3)&\text{for }x\in\mathbb{R}^2,t\in [1,\infty).
    \end{cases}
    \end{align}
    For this function, almost every level set is a Brakke flow, but it is not the unique viscosity solution $x_1^2+x_2^2+2t$. This shows that one does not have a comparison principle or uniqueness for Brakke flows.
\end{remark*}

The main ingredient for the proof of Theorem~\ref{Thm:MainThm} is the following comparison result for variational solutions.

\begin{lemma}[Comparison principle for variational solutions]
\label{Lem:CompPrinSet}
    Let $E(t)$, $F(t)$ be two variational solutions, cf.\ Definition~\ref{Def:VarSol}. Moreover, let them initially be ordered $E_0\subset F_0$ with $d(\partial^*E_0,\partial^*F_0)>0$. Then, for a.e.\ time $t\in (0,T)$ we have
    $$E(t)\subset F(t)\text{ and }d(\partial^*E(t),\partial^*F(t))>0,$$
    which means that the evolutions remain ordered.
\end{lemma}

\begin{figure}[!ht]
    \centering
    \includegraphics[width=0.35\linewidth]{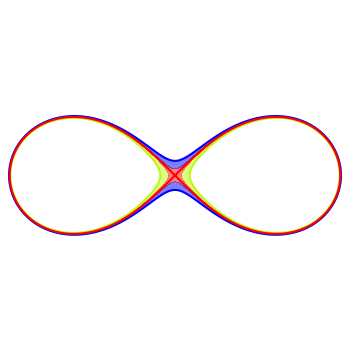}
    \includegraphics[width=0.35\linewidth]{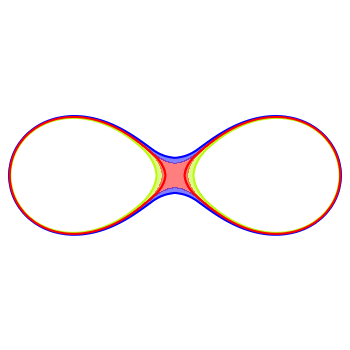}
    \caption{The lemniscate, an example of a curve which admits a non-unique variational solution in $\mathbb{R}^2$.}
    \label{Fig:Fattening}
\end{figure}

\begin{remark}
\label{Rem:Counterexample}
    Note that this result is sharp. Indeed, the conclusion $E(t)\subset F(t)$ does not hold in general if we assume only $E_0\subset F_0$. Take for example the lemniscate, i.e., the $0$ level set of the function $f(x,y)=(x^2+y^2)^2-x^2+y^2$ in $\mathbb{R}^2$ with a suitable cutoff at $-\frac{1}{5}$ and $1$ to ensure that it is well-prepared, see Figure~\ref{Fig:Fattening}. The blue and green solutions mark comparisons to show the maximal and minimal evolutions the curve can take. For the same initial datum, there exist two different variational solutions.
    This example is based on~\cite{angenent2002fattening}.
\end{remark}

\section{Proof of \texorpdfstring{Theorem~\ref{Thm:MainThm}}{the Main Theorem}}

\label{Sec:ProofMainThm}

Now we are in a position to prove our main theorem. The proof uses our comparison principle from Lemma~\ref{Lem:CompPrinSet}. This lemma implies that the evolution remains a function. Theorem~\ref{Thm:ViscBV} and Theorem~\ref{Thm:ViscBrakke} allow us to apply the lemma once more to compare the variational solution to the viscosity solution and we conclude that the evolution has to be the unique viscosity solution.

\begin{proof}[Proof of Theorem~\ref{Thm:MainThm}]
    Let $u$ be the viscosity solution starting at $g$. We show that we can construct a function $v$ from the evolutions $\Omega_s(t), \Sigma_s(t)$ of the (sub) level sets $\{g<s\}, \{g=s\}$ and that this function coincides with $u$.

    We note that for a.e.\ $s\in\mathbb{R}$, $\Omega_s^u(t)=\{u(\cdot,t)<s\}$ is a distributional solution and $\Sigma_s^u(t)=\{u(\cdot, t)=s\}$ can be seen as a Brakke flow. This is due to Theorem~\ref{Thm:ViscBV} and Theorem~\ref{Thm:ViscBrakke}. In the following, we restrict ourselves to the common values $s$ for which we have $\Omega_s^{u}(t), \Sigma_s^{u}(t)$ as these weak evolutions. Additionally, we require that these values are values for which we have $\Omega_s(t), \Sigma_s(t)$ as the evolutions in the theorem. Note that we have for a.e.\ $s,t$ that $\Sigma_s(t)=\partial^*\Omega_s(t)$ and $\Sigma_s^u(t)=\partial^*\Omega_s^u(t)$ by Remark~\ref{Rem:Boundary}. 

    Let $\sigma,s\in\mathbb{R}$ be values as above with $\sigma<s$. We can now apply Lemma~\ref{Lem:CompPrinSet} with $E=\Omega_\sigma(0)$ and $F=\Omega_s(0)$ since $E\subset F$. This shows that the sets remain ordered and thus can be interpreted as sub level sets of a function.
    
    Now, we can construct a function $v$ that is continuous in space and $L^1$-continuous in time via the layer cake formula:
    $$v(x,t)\coloneqq K-\int\limits_{-K}^K \chi_{\Omega_{s}(t)}(x)\dd s$$
    where $K\in \mathbb{R}$ is large enough and constant. Here, the continuity in space is a direct sequence of the comparison principle obtained by Lemma~\ref{Lem:CompPrinSet}: One can compare $v$ with a shifted version of $v$ and even obtain Lipschitz continuity.
    
    The continuity in time follows from Lemma~\ref{Lemma:L1ContBV}, the well-preparedness of $g$ and the estimate
    $$\int\limits_{\mathbb{R}^d}|v(x,t+h)-v(x,t)|\dd x\leq \int\limits_{-K}^K |\Omega_s(t+h)\Delta \Omega_s(t)|\dd s\to 0.$$

    Furthermore, we compare this function $v$ to the viscosity solution $u$ using again Lemma~\ref{Lem:CompPrinSet}. At $t=0$, both functions are the same and thus for all $\varepsilon>0$, we have 
    $$v(\cdot,0)-\varepsilon<u(\cdot,0)<v(\cdot,0)+\varepsilon.$$
    This also implies an ordering of the sub level sets
    $$\Omega_{s-\varepsilon}^{u}(0)\subsetneq \Omega_{s}(0)\subsetneq\Omega_{s+\varepsilon}^{u}(0).$$
    If we now choose $\varepsilon$ in such a way that both, $s$ and $s+\varepsilon$ are as above (resp.\ $s-\varepsilon$), we can apply the comparison principle, Lemma~\ref{Lem:CompPrinSet}, and get that their evolutions remain ordered. Since this is valid for a.e.\ level value, we conclude that the evolutions remain ordered:
    $$v(\cdot,t)\leq u(\cdot,t)+\varepsilon$$
    for almost all times $t\in (0,T)$. One can argue similarly for $v(\cdot, t)$ and $u(\cdot,t)-\varepsilon$ which yields
    $$|u-v|\leq \varepsilon \quad \text{for almost every }t\in(0,T).$$
    Letting $\varepsilon$ go to $0$, we conclude that the evolutions are the same. Thus, $v$ has a representative that is continuous in time.
\end{proof}

Now we prove the key ingredient, Lemma~\ref{Lem:CompPrinSet}. The proof is based on the time-continuity of distributional solutions, Lemma~\ref{Lemma:L1ContBV}, together with the avoidance principle of Brakke flows, Proposition~\ref{Prop:Avoidance}.

The main idea is the following. Since both variational solutions are Brakke flows and have a positive distance, the distance remains positive for all times. This is due to the avoidance principle, Proposition~\ref{Prop:Avoidance}. Thus, it remains to show that they don't overtake each other, mass is created or they change their ordering in any other way. This is done by the continuity coming from the $L^2$-bound of the normal velocity.

\begin{proof}[Proof of Lemma~\ref{Lem:CompPrinSet}]
    The proof is carried out using the following strategy. We prove the assertion via a contradiction and assume that $F$ overtakes $E$, i.e., parts of $E(t)$ lie outside of $F(t)$ for some time $t$. By the continuity in time, Lemma~\ref{Lemma:L1ContBV}, this implies that the boundaries have to intersect. This is shown with the construction of a suitable test function for the Brakke inequality and finished via a topological argument to get the contradiction.

    In the following, we work with good representatives of $E$ and $F$ since they are sets of finite perimeter, cf.~\cite[Proposition 12.19]{maggi2012sets}. By this proposition, one can change a set of finite perimeter up to a null set such that the topological boundary is the support of the Gauss-Green measure, i.e., 
    $$\partial E=\supp \mu_E=\left\{x\in\mathbb{R}^d:\frac{|E\cap B_r(x)|}{\omega_dr^d}\in (0,1)\ \forall r>0\right\}.$$
    
    By assumption, there exists a constant $c_1>0$ such that 
    $$E_0\subseteq F_0\text{ and } d(E_0,F_0)>c_1.$$
    Therefore, we conclude that for all times $t>0$ the distance remains bounded below $$d(E(t),F(t))>c_1$$ by the avoidance principle, cf.~Proposition~\ref{Prop:Avoidance}. Hence, we also have 
    $$\overline E_0\subset F_0.$$

    We assume that this ordering $E\subseteq F$ up to a null set does not remain for all times and show a contradiction. Then, there exists a time $t$ such that $|E(t)\setminus F(t)|>0$. Let $t_*$ be the infimum of these times. To be precise, one would take a sequence $t_{*,k}\nearrow t_*$ since the properties only hold for almost all times but for brevity we don't write this step out. Now without loss of generality we can set this $t_*$ to $0$ for the ease of notation. To indicate this change, we denote the small times after $t_*$ by $h$. With this notation we have $|E(0)\setminus F(0)|=0$. We know $|F(0)\setminus E(0)|>0$.
    
    Moreover, by the definition of $t_*$, there exists $h>0$ arbitrarily small with $|E(h)\setminus F(h)|>0,$ i.e., the ordering fails. Without loss of generality, we can assume that $|F(0)|\geq |E(0)|>0$ as otherwise $E$ would remain empty. Thus, we can choose $h$ small enough such that $|E(h)|,|F(h)|,|E(h)\cap F(h)|$ and $|E(h)\setminus F(h)|$ are positive, cf.\ Lemma~\ref{Lemma:L1ContBV}.
    
    Furthermore, we can define $f_E\in C^1_{t,x}$ such that 
    \begin{align}
        0\leq f_E\leq 1,\quad f_E&\equiv 0\text{ on }E^c,\nonumber\\
        \partial_t f_E,\nabla f_E&\equiv 0\text{ on }\{(x,t):d(x,\partial E(t))>c_1\},\label{Eq:fE}\\
        f_E&\equiv 1\text{ on }\{(x,t)\in E:d(x,\partial E(t))>c_1\}.\nonumber
    \end{align}
    This is done as follows. We define $E_{c_1}^-\coloneqq \{(x,t)\in E:d(x,\partial E(t))>c_1\}$, $E_{c_1}^+\coloneqq \left\{(x,t)\in E^c:d(x,\partial E(t))<c_1\right\}\cup E$ and set $f_E$ to be equal to $1$ on $E_{c_1}^-$ and $0$ on $(E_{c_1}^+)^c$. We now show that there exist two smooth sets $E^\varepsilon_{1/2}$ s.t.\ we have the following ordering for $t\in [0,h]$
    \begin{align*}
        E_{c_1}^-&\subset E_1^\varepsilon\subset E \subset E_2^\varepsilon\subset E_{c_1}^+
    \end{align*}
    with a positive distance of their boundaries. Then, we can smoothly interpolate from $1$ to $0$ between $E^\varepsilon_{1}$ and $E^\varepsilon_{2}$.

    First, we define sets whose boundaries are parallel sets to the boundary of $E$. These can also be seen as the boundary of the (space-time) tubular neighborhood of the boundary of $E$. The sets are $\widetilde E_1=E\setminus \{(x,t)\in E: d((x,t),\partial E)<\delta\}$ and $\widetilde E_2=\{(x,t)\in E^c: d(x,\partial E(t))<\delta\}\cup E$.

    By definition, we have $\widetilde E_1\subset E$ and the distance of $\partial \widetilde E_1$ to $\partial E$ is precisely $\delta$. Analogously, $E\subset \widetilde E_2$ and for all $x\in\partial \widetilde E_2$, we have $d(x,\partial E)=\delta$.

    The boundaries $\partial\widetilde E_{1/2}$ are on each side of $\partial E$ and are sets of finite perimeter by \cite[Theorem 3.1]{ChambolleTubularNbhd} since $\overline{E}$ is compact.
    
    Now, we show that $\partial \widetilde E_1$ and $\partial E_{c_1}^-$ have a distance that is bounded uniformly from below. The similar assertion for $\partial \widetilde E_2$ holds analogously.
    
    For this, we compare the evolution of $\partial E$ in time with the evolution of spheres of radius less than $c_1$ centered at $\partial E_{c_1}^-$, cf.\ Lemma~\ref{Lem:CompPrinSet}. This implies that there exist $\widetilde t>\delta$ (for $\delta$ chosen sufficiently small only depending on $c_1$ and the dimension $d$) and $\widetilde c<c_1-\delta$ such that for all $\tau<\widetilde t$
    $$\partial E(t+\tau)\subset \partial E(t)+B_{\widetilde c}.$$
    Thus, by the definition of $E_{c_1}^-$ and the triangle inequality, we also have 
    $$d(\partial E(t+\tau), \partial E_{c_1}^-(t))\geq c_1-\widetilde c.$$
    Since $d(\partial\widetilde E_1,\partial E)\leq \delta$, this then implies that for any $x\in \partial \widetilde E_1(t)$ there exists $\tau\leq \delta$ such that $d(x,\partial E(t+\tau))=\delta$. Thus, we have 
    $$d(x,\partial E_{c_1}^-(t))\geq d(\partial E(t+\tau),\partial E_{c_1}^-(t))-d(x,\partial E(t+\tau))\geq c_1-\widetilde c-\delta>0.$$
    Hence, we obtain a uniform lower bound of the distance between the boundaries $\partial \widetilde E_1$ and $E_{c_1}^-$.

    In case the original $t^*$ is too small to allow for the times $t^*-\delta$, we extend the evolution of $E$ as a constant for times less than $0$.

    Since all boundaries have a positive distance and $\widetilde E_{1/2}$ are sets of finite perimeter in space and time, we can approximate $\widetilde E_{1/2}$ by smooth sets $E^\varepsilon_{1/2}$ satisfying our conditions.

    Hence, we can smooth these sets out with a mollifying kernel. For the space-time ball $B^{d+1}_1$, consider $\varphi\in C_c^\infty(B^{d+1}_1), \int\int \varphi\dd x\dd t=1, \varphi\geq 0$ and define $\varphi_\varepsilon\coloneqq \frac{1}{\varepsilon^{d+1}}\varphi\left(\frac{\cdot}{\varepsilon}\right)$. Then by~\cite[Theorem~13.8]{maggi2012sets}, the sets $E_{1/2}^\varepsilon\coloneqq \{\varphi_\varepsilon*\chi_{\widetilde E_{1/2}}\geq s_\varepsilon\}$ for suitable $s_\varepsilon\in (0,1)$ converge to $\widetilde E_{1/2}$ in measure and are smooth.
    Moreover, since $\supp \varphi_\varepsilon \subset B^{d+1}_\varepsilon$, we have $\partial  E^\varepsilon_{1/2} \subset \partial \widetilde E_{1/2}+B^{d+1}_\varepsilon$. Hence, for $\varepsilon$ small enough, the ordering from above remains. This now allows to smoothly interpolate $f_E$ from 1 in $E_1^\varepsilon$ to $0$ outside of $E_2^\varepsilon$.

    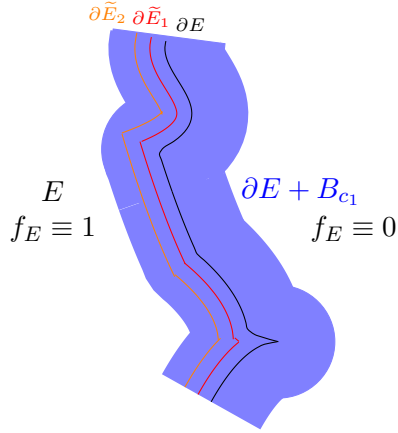
\begin{figure}
        \centering
        \begin{tikzpicture}
          \draw[scale=4, domain=-0.2:0, smooth, variable=\x, line width=1.5cm, mark=circle, blue!50]  plot ({-0.5*sqrt(-\x)},{\x});
          \draw[scale=4, domain=0:0.5, smooth, variable=\x, line width=1.5cm, blue!50]  plot ({max(-0.5*sqrt(\x),-2*\x*\x-0.1)},{\x});
          \draw[scale=4, domain=0.5:1, smooth, variable=\x, line width=1.5cm, blue!50]  plot ({max(-0.5*sqrt(\x),sin(deg(5*(\x-1.5)))+tan(deg(\x-1.5))+cos(deg(3*(\x-1.5)))+0.7)},{\x});
    
          \filldraw [blue!50] (0,0) circle (0.75cm);
          \filldraw [blue!50] (-1.6,2.56) circle (0.75cm);
          \filldraw [blue!50] (-2.196,4.16) -- (-2.196,4.16) arc(170:355:0.75cm) --cycle;

          {
            \draw[scale=4, domain=-0.2:-0.055, smooth, variable=\x, red]  plot ({-sqrt(-\x)/2-0.075*4*sqrt(-\x)/(sqrt(1-16*\x)},{\x+0.075*1/sqrt(1-16*\x)});
            \draw[scale=4, domain=0.055:0.06, smooth, variable=\x, red]  plot ({-sqrt(\x)/2-0.075*4*sqrt(\x)/(sqrt(1+16*\x)},{\x-0.075*1/sqrt(1+16*\x)});
            \draw[scale=4, domain=0.01:0.307, smooth, variable=\x, red]  plot ({-2*\x^2-0.1-0.075*1/(sqrt(1+16*\x*\x)},{\x-0.075*4*\x/sqrt(1+16*\x*\x)});
            \draw[scale=4, domain=0.28:0.7, smooth, variable=\x, red]  plot ({-sqrt(\x)/2-0.075*4*sqrt(\x)/(sqrt(1+16*\x)},{\x-0.075*1/sqrt(1+16*\x)});
            \draw[scale=4, domain=0.607:1, smooth, variable=\x, red]  plot ({sin(deg(5*(\x-1.5)))+tan(deg(\x-1.5))+cos(deg(3*(\x-1.5)))+0.7-0.075*1/sqrt(1+(5*cos(deg(5*(\x-1.5)))+1+(tan(deg(\x-1.5)))^2-3*sin(deg(3*(\x-1.5))))^2)},{\x+0.075*(5*cos(deg(5*(\x-1.5)))+1+(tan(deg(\x-1.5)))^2-3*sin(deg(3*(\x-1.5))))/sqrt(1+(5*cos(deg(5*(\x-1.5)))+1+(tan(deg(\x-1.5)))^2-3*sin(deg(3*(\x-1.5))))^2)}) node[above left] {\tiny$\partial \widetilde E_1$};
          }

          {
            \draw[scale=4, domain=-0.2:0, smooth, variable=\x, orange]  plot ({-sqrt(-\x)/2+0.075*4*sqrt(abs(\x))/(sqrt(1+16*abs(\x))},{\x-0.075/(sqrt(1+16*abs(\x))});
            \draw[scale=4, domain=0:0.02, smooth, variable=\x, orange]  plot ({-sqrt(\x)/2+0.075*4*sqrt(abs(\x))/(sqrt(1+16*abs(\x))},{\x+0.075/(sqrt(1+16*abs(\x))});
            \draw[scale=4, domain=-90:90, smooth, variable=\x, orange]  plot ({0.075*cos(\x)},{0.075*sin(\x)});
            \draw[scale=4, domain=0.065:0.28, smooth, variable=\x, orange]  plot ({-2*\x^2-0.1+0.075*1/(sqrt(1+16*\x*\x)},{\x+0.075*4*\x/sqrt(1+16*\x*\x)});
            \draw[scale=4, domain=0.3:0.58, smooth, variable=\x, orange]  plot ({-sqrt(\x)/2+0.075*4*sqrt(\x)/(sqrt(1+16*\x)},{\x+0.075*1/sqrt(1+16*\x)});
            \draw[scale=4, domain=0.665:1.0, smooth, variable=\x, orange]  plot ({sin(deg(5*(\x-1.5)))+tan(deg(\x-1.5))+cos(deg(3*(\x-1.5)))+0.7+0.075*1/sqrt(1+(5*cos(deg(5*(\x-1.5)))+1+(tan(deg(\x-1.5)))^2-3*sin(deg(3*(\x-1.5))))^2)},{\x-0.075*(5*cos(deg(5*(\x-1.5)))+1+(tan(deg(\x-1.5)))^2-3*sin(deg(3*(\x-1.5))))/sqrt(1+(5*cos(deg(5*(\x-1.5)))+1+(tan(deg(\x-1.5)))^2-3*sin(deg(3*(\x-1.5))))^2)}) node[above right] {\tiny$\partial \widetilde E_2$};
          }
        
          \draw[scale=4, domain=-0.2:0, smooth, variable=\x, black]  plot ({-0.5*sqrt(-\x)},{\x});
          \draw[scale=4, domain=0:0.5, smooth, variable=\x, black]  plot ({max(-0.5*sqrt(\x),-2*\x*\x-0.1)},{\x});
          \draw[scale=4, domain=0.5:1, smooth, variable=\x, black]  plot ({max(-0.5*sqrt(\x),sin(deg(5*(\x-1.5)))+tan(deg(\x-1.5))+cos(deg(3*(\x-1.5)))+0.7)},{\x}) node[above] {\tiny$\partial E$};
          \node at (-3,1.5) {$f_E\equiv 1$};
          \node at (-3,2) {$E$};
          \node at (1,1.5) {$f_E\equiv 0$};
          \node [blue] at (0.3,2) {$\partial E+B_{c_1}$};
        \end{tikzpicture}
        \vspace*{-2em}
        \caption{Visualization of a time slice of the interpolation construction.}
        \label{fig:placeholder}
    \end{figure}

    Using this function $f_E$ in the definition of Brakke flows for $\partial F$, by~\eqref{Eq:fE} we derive (the measure of the varifold agrees with the total variation measure of the Gauss-Green measure)
    $$\int\limits_{\mathbb{R}^d}f_E(\cdot, h)\dd|\mu_{F(h)}|+\int\limits_0^h\int\limits_{\mathbb{R}^d}|\vec{H}_{F(t)}|^2f_E(\cdot, t)\dd|\mu_{F(t)}|\dd t\leq \int\limits_{\mathbb{R}^d}f_E(\cdot, 0)\dd|\mu_{F(0)}|=0.$$
    The right-hand side is zero as $0\leq f_E\leq \chi_E$ and $\overline{E}(0)\subseteq F(0)$ which implies that $f_E$ is supported away from $\partial F$. This implies that
    $$\int\limits_{\mathbb{R}^d}f_E(\cdot, h)\dd|\mu_{F(h)}|=0.$$
    This means that there is no boundary of $F(h)$ in $E(h)$. This can be seen as follows. By the distance between the boundaries coming from the avoidance principle, Proposition~\ref{Prop:Avoidance}, the boundary of $F(h)$ cannot be $c_1$-near the boundary of $E(h)$. Furthermore, outside of this neighborhood to the boundary of $E(h)$, $f_E$ is equal to $1$. If there would be a point $x\in \{x\in E(h): d(x,\partial E(h))>c_1\}$ of $\partial F(h)$, we can find an $r$ s.t.\ $B_r(x)$ is contained in $E(h)$. Thus, we would have a density of $F(h)$ in $B_r(x)$ between $0$ and $1$ around this point since we chose a good representative for $F$. This in turn would imply that the integral is positive by the (relative) isoperimetric inequality. Thus, we conclude
    $$\partial F(h)\cap E(h)=\emptyset.$$

    Analogously, we have $\partial E(h)\cap F(h)^c=\emptyset$ since for $f_F$ similar to above with exchanged roles for $E(h)$ and $F(h)$, we have 
    $\int_{\mathbb{R}^d}f_F(\cdot, h)\dd|\mu_{E(h)}|=0$
    and thus there cannot be a point of the reduced boundary of $E(h)$ in $F(h)^c$.

    Thus, we have shown under our assumption that $\partial E(h)\cap F(h)^c=\emptyset$, $\partial F(h)\cap E(h)=\emptyset$, $|E(h)\setminus F(h)|>0, |E(h)\cap F(h)|>0$. Now, we can choose an $x\in E(h)\setminus F(h)$ and look at the component $Z^{E(h)}_x$ of $E(h)$ containing $x$. Since this component is connected and $F(h)$ has no boundary in $E(h)$, we have that $Z^{E(h)}_x\cap F(h)$ has to be empty or equal to $Z^{E(h)}_x$. The latter is not possible by the existence of $x$. Thus, every component of $E(h)$ is either contained in $F(h)$ or disjoint with $F(h)$ and there exists at least one disjoint component $Z^{E(h)}_x$. Furthermore, we can look at the component of $x$ in $F(h)^c$, $Z^{{F(h)}^c}_x$. Since $E(h)$ has empty boundary in $F(h)^c$, we have that $E(h)\cap Z^{F(h)^c}_x$ has to be empty or equal to $Z^{F(h)^c}_x$. Because of $x$ it cannot be empty. If it were the whole set, i.e., $Z^{F(h)^c}_x\subset E(h)$, there would be boundary of $F(h)^c$ and thus boundary of $F(h)$ in $E(h)$.
    
    This proves our contradiction and thus the ordered sets $E$ and $F$ remain ordered.
\end{proof}

Thus, we now have shown the equivalence of these weak solution concepts. Namely, given a well-prepared initial datum, it is equivalent to say that an evolution is the unique viscosity solution and to say that almost every sub level set is a variational solution. This connects the solution concepts based on gradient flows to the comparison based concepts and gives the reverse direction to the results of \cite{ESIV, genericMCF}.

\section{Generic uniqueness of variational solutions}
Our methods also allow us to prove the generic uniqueness of variational solutions. This is done by comparing the given solution to the variational solution coming from the viscosity solution.

\begin{theorem}[Uniqueness of variational solution]
    \label{Thm:UniqueVarSol}
    Let $g$ be a well-prepared initial datum in the sense of Definition~\ref{def:wellprepared}. Then, almost every level set has a unique variational solution.
\end{theorem}

\begin{remark}
    This theorem is optimal in the sense that one cannot drop the genericity assumption, i.e., that the statement holds for almost every level set. This can be seen from the lemniscate, see Figure~\ref{Fig:Fattening}. For the viscosity solution, the level set fattens which is shown in red. Here, different evolutions are feasible as described before in Remark~\ref{Rem:Counterexample}.
\end{remark}

\begin{proof}
    Since $g$ is well-prepared, we can take the unique viscosity solution $u$ and almost every level set of $u$ is a variational solution. Take such a level set value and consider a possibly different variational evolution of the sub level set $\Omega$.
    
    Here, we can apply Lemma~\ref{Lem:CompPrinSet}. We compare $\Omega(t)$ to the variational solution coming from the viscosity solution $\Omega^u_s(t)$. At time $0$ both sets agree. Thus, we can find a sequence of $\varepsilon$ such that 
    $$\Omega\subset \Omega^u_{s+\varepsilon}$$
    for all times $t$ and all $\varepsilon$ and such that $\Omega^u_{s+\varepsilon}$ is a variational solution, similarly for $s-\varepsilon$. Moreover, we again have $\Sigma^u_{s+\varepsilon}=\partial^*\Omega^u_{s+\varepsilon}$ and $\Sigma=\partial^*\Omega$. Together, this implies
    $$\bigcup\limits_{\varepsilon>0}\Omega_{s-\varepsilon}^u\subset \Omega\subset\bigcap\limits_{\varepsilon>0}\Omega_{s+\varepsilon}^u.$$
    Therefore, we conclude $\Omega=\Omega_s^u$ and $\Sigma=\Sigma_s^u$ for all almost every time which especially implies uniqueness of variational solutions.
\end{proof}

Theorem~\ref{Thm:UniqueVarSol} complements the weak-strong uniqueness result of distributional solutions in \cite[Theorem 1]{FHLS}. The latter states in particular that variational solutions are unique before the onset of singularities. Our result shows the uniqueness past generic singularities.

\frenchspacing
\bibliographystyle{plainurl}
\bibliography{References.bib} 
\end{document}